\documentclass{gtpart}
\usepackage[all]{xy}
\title[Formality of Sinha's cosimplicial space and Gerstenhaber algebras]{Formality of Sinha's cosimplicial model for long knots spaces and the Gerstenhaber algebra structure of homology}

\author{Paul Arnaud Songhafouo Tsopméné}
\givenname{Paul Arnaud}
\surname{Songhafouo Tsopméné}
\address{University of Regina,  3737 Wascana Pkwy, Regina, SK S4S 0A2, Canada}
\email{pso748@uregina.ca}
\urladdr{}


\arxivreference{arXiv:1210.2561v2}
\arxivpassword{}

\volumenumber{}
\issuenumber{}
\publicationyear{}
\papernumber{}
\startpage{}
\endpage{}
\doi{}
\MR{}
\Zbl{}
\received{}
\revised{}
\accepted{}
\published{}
\publishedonline{}
\proposed{}
\seconded{}
\corresponding{}
\editor{}
\version{}

\newtheorem{thm}{Theorem}[section]    
\newtheorem{lem}[thm]{Lemma}          
\theoremstyle{definition}
\newtheorem{defn}[thm]{Definition}    
\newtheorem{expl}[thm]{Example}  
\newtheorem{coro}[thm]{Corollary}
\newtheorem{prop}[thm]{Proposition}
\newtheorem{rem}[thm]{Remark} 
\begin{document}

\begin{abstract}    
Sinha constructed a cosimplicial space $\mathcal{K}_N^{\bullet}$ that gives a model for the space of long knots modulo immersions in $\mathbb{R}^N$, $N \geq 4$. On the other hand, Lambrechts, Turchin and Voli\'c showed that for $N \geq 4$ the homology Bousfield-Kan spectral sequence associated to Sinha's cosimplicial space $\mathcal{K}_N^{\bullet}$ collapses at the $E^2$ page rationally. Their approach  consists in first proving the formality of some other diagrams approximating $\mathcal{K}_N^{\bullet}$ and next deducing the collapsing result. In this paper, we prove directly the formality of Sinha's cosimplicial space, which immediately implies the collapsing result for $N \geq 3$. Moreover, we prove that the isomorphism between the $E^2$ page and the homology of the space of long knots modulo immersions respects the Gerstenhaber algebra structure, when $N \geq 4$. 
\end{abstract}

\maketitle


\section{Introduction}

A \textit{multiplicative operad} in a symmetric monoidal category $\mathcal{C}$ consists of a couple $(\mathcal{O}, \alpha)$ in which $\mathcal{O}$  is a non-symmetric operad in $\mathcal{C}$ and $\alpha$ is a morphism from the associative operad $\mathcal{A}s$ (see Example~\ref{as_example} below for its definition) to $\mathcal{O}$. The multiplicative operad that we look at in this paper is Kontsevich's operad $\mathcal{K}_N=\{\mathcal{K}_N(n)\}_{n \geq 0}$, which was defined and studied by Sinha in \cite[Definition 4.1 and Theorem 4.5]{sin06}. Notice that this operad is equivalent to the little $N$-disks operad. Notice also that  to any multiplicative operad $\mathcal{O}$, one can associate a cosimplicial object $\mathcal{O}^{\bullet}$ (see \cite[Section 10]{mcc_smith04}). Hence  the multiplicative operad $\mathcal{K}_N$ induces a cosimplicial space, $\mathcal{K}_N^{\bullet}$, called \textit{Sinha's cosimplicial space}.

For $N \geq 3$, a \textit{long knot} is a smooth embedding $\mathbb{R} \hookrightarrow \mathbb{R}^N$ that coincides outside a compact set with a fixed linear embedding. The \textit{space of long knots modulo immersion}, denoted by $\overline{\mbox{Emb}}(\mathbb{R}, \mathbb{R}^{N})$, is defined as the homotopy fiber of the inclusion of the space of long knots in the space of long immersions. In \cite{sin06} Sinha proves that for $N \geq 4$, the homotopy totalization of $\mathcal{K}_N^{\bullet}$ is weakly equivalent to the space of long knots modulo immersion,
$$\mbox{hoTot}\mathcal{K}_N^{\bullet} \simeq \overline{\mbox{Emb}}(\mathbb{R}, \mathbb{R}^N).$$
He also proves that for $N \geq 4$, the homology Bousfield-Kan spectral sequence associated to $\mathcal{K}_N^{\bullet}$ converges to the homology $H_*(\mbox{hoTot}\mathcal{K}_N^{\bullet}) \cong  H_*(\overline{\mbox{Emb}}(\mathbb{R}, \mathbb{R}^{N}))$ \cite[Theorem 7.2]{sin06}. Therefore it is natural to ask whether this spectral sequence collapses or not. This question was studied by Lambrechts, Turchin and Voli\'c in \cite{lam_tur_vol10}, who proved the following result.
\begin{thm}\emph{\cite[Theorem 1.2]{lam_tur_vol10}} \label{collapsing_thm}
For $N \geq 3$ the homology Bousfield-Kan spectral sequence associated to Sinha's cosimplicial space $\mathcal{K}_N^{\bullet}$ collapses at the $E^2$ page rationally.
\end{thm}
\begin{rem}
Actually in \emph{\cite{lam_tur_vol10}} \emph{Theorem~\ref{collapsing_thm}} is proved only for $N > 3$, but our approach also does the work for $N=3$.
\end{rem}
The authors of \cite{lam_tur_vol10} prove Theorem~\ref{collapsing_thm} without showing that $\mathcal{K}_N^{\bullet}$ is formal. The main ingredient of their proof is the relative version of Kontsevich's theorem on the formality of the little $N$-disks operad \cite{lam_vol}. This theorem states that there exists a chain of quasi-isomorphisms of operads between the singular chains on the little $N$-disks operad and its homology with real coefficients. In particular, there is a zig-zag 
$$S_*(\mathcal{K}_N) \stackrel{\sim}{\longleftarrow} \cdots \stackrel{\sim}{\longrightarrow} H_*(\mathcal{K}_N)$$
in which the chain complex of the Fulton-MarcPherson operad $S_*(\mathcal{F}_N)$ appears. This poses a serious problem to the authors of \cite{lam_tur_vol10} because the operad $S_*(\mathcal{F}_N)$ is not multiplicative, but only multiplicative "up to homotopy" in the sense of Definition~\ref{multiplicative_def}. This problem is solved by introducing certain finite diagrams of spaces called \textit{fanic diagrams}. 

Here are our results.
\begin{thm}\label{formality_thm} For $N \geq 3$, the operads $S_* (\mathcal{K}_N; \mathbb{R})$ and $H_* (\mathcal{K}_N; \mathbb{R})$ are weakly equivalent as multiplicative operads. 
\end{thm}

For the meaning of "weakly equivalent as multiplicative operads", see Definition~\ref{multiplicative_def}. 
\begin{rem} In \emph{\cite{lam_vol}} it is only proved that $S_*(\mathcal{K}_N; \mathbb{R})$ and $H_*(\mathcal{K}_N; \mathbb{R})$ are weakly equivalent as "up to homotopy multiplicative operads" \emph{(Definition~\ref{multiplicative_def})}, when $N \geq 3$. Notice that this result is not proved for $N =2$  but only for $N \geq 3$ (see \cite[Theorem 1.4]{lam_vol}).
\end{rem}
An immediate consequence of Theorem~\ref{formality_thm} is the following formality result.
\begin{coro}\label{main_result}
For $N \geq 3$ Sinha's cosimplicial space $\mathcal{K}_N^{\bullet}$ is formal over $\mathbb{R}$.  
\end{coro}
Our method enables us also to determine the Gerstenhaber structure on the homology of the space of long knots.

We explain now with which  Gerstenhaber structures we endow $H_*(\overline{\mbox{Emb}}(\mathbb{R}, \mathbb{R}^N); \mathbb{R})$ and  $HH(H_*\mathcal{K}_N; \mathbb{R})$.

McClure and Smith construct in \cite{mcc_smith04} two operads, $\mathcal{D}_2$ and $\widetilde{\mathcal{D}}_2$, both weakly equivalent to the little $2$-disks operad $\mathcal{C}_2$. They show that if a cosimplicial space $\mathcal{O}^{\bullet}$ is built from a multiplicative operad $\mathcal{O}$, then $\mathcal{D}_2$ acts on the totalization $\mbox{Tot}\mathcal{O}^{\bullet}$ \cite[Theorem 9.1]{mcc_smith04}, and $\widetilde{\mathcal{D}}_2$ acts on the homotopy totalization $\mbox{hoTot} \mathcal{O}^{\bullet}$ \cite[Theorem 15.3]{mcc_smith04}. 
If in addition $\mathcal{O}$ is reduced (that is,  both $\mathcal{O}(0)$ and $\mathcal{O}(1)$ are weakly contractible), then the homotopy totalization of $\mathcal{O}^{\bullet}$ is weakly homotopy equivalent to the double loop space of a certain explicit space of maps of operads (Dwyer-Hess \cite{dwyer_hess10} and Turchin \cite{tur10} prove this result by using different approaches). 
Notice that neither Dwyer-Hess nor Turchin actually prove that their delooping is the delooping with respect to the McClure-Smith $\widetilde{\mathcal{D}}_2$ action.

Let us come back now to the particular case of Kontsevich's operad $\mathcal{K}_N$, which is reduced. Since it is multiplicative, it follows that the operad $\widetilde{\mathcal{D}}_2$ acts on $\mbox{hoTot}\mathcal{K}_N^{\bullet} \simeq \overline{\mbox{Emb}}(\mathbb{R}, \mathbb{R}^N)$. We also have a $\mathcal{C}_2$ geometric action (constructed by Budney in \cite{bud07}) on (framed) long knots. One question arises: are these two actions equivalent? This question is still open to my knowledge. But one thing is certain: each of these actions induces a Gerstenhaber algebra structure on $H_*(\overline{\mbox{Emb}}(\mathbb{R}, \mathbb{R}^N); \mathbb{R})$, and apparently it has never been checked whether these structures coincide. We now specify which one we choose. 

\textbf{(A)} The left hand side of equation (\ref{main_equation}) is equipped with the Gerstenhaber algebra structure induced by the action of $\widetilde{\mathcal{D}}_2$ on $\mbox{hoTot}\mathcal{K}_N^{\bullet} \simeq \overline{\mbox{Emb}}(\mathbb{R}, \mathbb{R}^N)$.


On the other hand, associated to a multiplicative operad $\mathcal{B}_*(\bullet)$ in chain complexes is its Hochschild homology $HH(\mathcal{B}_*(\bullet))$, defined first by Gerstenhaber and Voronov in \cite{ger_vor95}. It is endowed with a natural Gerstenhaber algebra structure (see \cite{ger_vor95} or \cite[Section 4]{sal06} for more details about this natural structure). 

\textbf{(B)} The right hand side of equation (\ref{main_equation}), $HH(H_*\mathcal{K}_N; \mathbb{R})$, is equipped with the natural Gerstenhaber algebra structure.
\begin{coro}\label{main_theorem}
For $N \geq 4$,  there exists an isomorphism of \emph{Gerstenhaber} algebras between the homology of space of long knots modulo immersions and the Hochschild homology $HH(H_*\mathcal{K}_N)$ over $\mathbb{R}$ when the homology $H_*(\overline{\mbox{Emb}}(\mathbb{R}, \mathbb{R}^N); \mathbb{R})$ is equipped with the Gerstenhaber algebra structure described by \emph{\textbf{(A)}}, and $HH(H_*\mathcal{K}_N; \mathbb{R})$ is equipped with the one  described by \emph{\textbf{(B)}}. That is,
\begin{eqnarray}\label{main_equation}
H_*(\overline{\mbox{Emb}}(\mathbb{R}, \mathbb{R}^N); \mathbb{R}) \cong HH(H_*\mathcal{K}_N; \mathbb{R}).
\end{eqnarray}
\end{coro}

\begin{rem}
For $N \geq 4$, it is proved in \emph{\cite{lam_tur_vol10}} that $HH(H_*\mathcal{K}_N; \mathbb{Q})$ and $H_*(\overline{\mbox{Emb}}(\mathbb{R}, \mathbb{R}^N); \mathbb{Q})$ are isomorphic as vector spaces but not as \emph{Gerstenhaber} algebras.
\end{rem}
\begin{rem}\label{sakai_rem} In \emph{\cite[Theorem 2.3]{sakai08}} \emph{Sakai} proves a result, announced earlier by \emph{Salvatore} \emph{\cite[Proposition 22]{sal06}}, which states that the graded vector spaces $HH(S_*\mathcal{K}_N)$ and $H_*(\overline{\mbox{Emb}}(\mathbb{R}, \mathbb{R}^N))$  are isomorphic as \emph{Gerstenhaber} algebras. 
\end{rem}
\begin{rem}
When a version of this paper was ready, \emph{Syunji Moriya} put in arXiv a paper \emph{\cite{moriya12}} in which equivalent results are independently discovered.
\end{rem}

\textbf{Outline of the paper.}\\
In Section~\ref{nonsym_operads} we first recall the definition of a non-symmetric operad. Next we show that the axiom of relative properness holds in the category of non-symmetric operads (this is just the non-symmetric version of \cite[Theorem 12.2.B]{fress08}). Finally we prove Lemma~\ref{main_lemma}, which is crucial for the rest of the paper.\\
In Section~\ref{formality_section} we apply Lemma~\ref{main_lemma} to the specific zig-zag between $S_*(\mathcal{K}_N; \mathbb{R})$ and its homology operad $H_*(\mathcal{K}_N; \mathbb{R})$, and we obtain the main result of the paper, which states that the Kontsevich's operad is formal over $\mathbb{R}$ as a multiplicative operad (Theorem~\ref{formality_thm}). This result implies immediately that Sinha's cosimplicial space is formal over $\mathbb{R}$ (Corollary~\ref{main_result}). Using now this formality, we give a very short proof of the collapse of the Vassiliev spectral sequence over rationals (Theorem~\ref{collapsing_thm}). We end the section with a proof of Corollary~\ref{main_theorem} which is also a consequence of Theorem~\ref{formality_thm}.

\textbf{Acknowledgements.} The author is grateful to Pascal Lambrechts for  suggesting the idea of left properness axiom and also for his encouragement. 
I also thank Yves Félix for giving me a copy of \cite{fress08}, which is a central reference for this paper. Obviously, I cannot forget to thank Benoît Fresse and Paolo Salvatore for answering all my emails with questions about the homotopy theory of operads and of cosimplicial spaces.

\section{The category of non-symmetric operads and equivalences of multiplicative operads}\label{nonsym_operads}
Throughout this section we denote by $\mathcal{C}\colon \!\!\!\!=(\mathcal{C}, \otimes, \mbox{\textbf{1}})$  a symmetric monoidal model category that is cofibrantly generated \cite[Section 11.1.5 and Section 11.3.3]{fress08}. The category $\mathcal{C}^{\mathbb{N}}$ of sequences $X=\{X(n)\}_{n \geq 0}$ in $\mathcal{C}$ will be equipped with the obvious model structure (that is, weak-equivalences, fibrations and cofibrations are all levelwise). Note that model categories for us are as in \cite{hovey99}.\\
We begin  by recalling the definition of a non-symmetric operad.  Next we state and prove the non-symmetric version of \cite[Theorem 12.2.B]{fress08}, which states that the axiom of relative properness  holds in the category of non-symmetric operads in $\mathcal{C}$.  We end with the crucial Lemma~\ref{main_lemma} for this paper.


A \emph{symmetric operad} in $\mathcal{C}$ consists of a symmetric sequence $\mathcal{O}=\{\mathcal{O}(n)\}_{n \geq 0}$ (for each $n$ the symmetric group $\Sigma_n$ acts on $\mathcal{O}_n$) of objects of $\mathcal{C}$ endowed with an unit element $\textbf{1} \longrightarrow \mathcal{O}(1)$ and a collection of morphisms
$$\mathcal{O}(k) \otimes \mathcal{O}(i_1) \otimes \cdots \otimes \mathcal{O}(i_k) \longrightarrow \mathcal{O}(i_1+\cdots+i_k)$$ 
that satisfy natural equivariance properties, unit and associative axioms $($May's axioms, see \cite{may72}$)$. 
If we omit the action of $\Sigma_n$, then  $\mathcal{O}$ is called a \textit{non-symmetric operad}. The category of non-symmetric operads in $\mathcal{C}$ will be denoted by $\mathcal{O}p_{ns}(\mathcal{C})$ or by the short notation $\mathcal{O}p_{ns}$. 
\begin{expl}\label{as_example}
Let $\mathcal{A}s=\{\mathcal{A}s(n)\}_{n \geq 0}$ be the sequence defined by $\mathcal{A}s(n)=\emph{\mbox{\textbf{1}}}$ for each $n$, the unit for the tensor product of $\mathcal{C}$. It is easy to see that $\mathcal{A}s$ is a non-symmetric operad, called the \emph{associative operad}.
\end{expl}
\begin{rem}\label{as_is_cofibrant}
The object $\mathcal{A}s$ is cofibrant in the category $\mathcal{C}^{\mathbb{N}}$ because the unit object \emph{\mbox{\textbf{1}}} is cofibrant in $\mathcal{C}$ by the unit axiom, which is a part of the definition \emph{(\cite[Section 11.3.3]{fress08})} of a symmetric monoidal model category. 
\end{rem}
Notice that the category $\mathcal{O}p_{ns}(\mathcal{C})$ usually has only semi-model structure, by the non-symmetric version of \cite[Theorem 12.2.A]{fress08}.
But, if $\mathcal{C}$ satisfies certain conditions \cite[Theorem 1.1]{muro11}, then $\mathcal{O}p_{ns}(\mathcal{C})$ turns out to be a model category. For example it is not difficult to see that the category of non-negatively graded chain complexes Ch$_{\mathbb{R}}$ satisfies such conditions. 
\begin{rem}\label{chain_model}
The category $\mathcal{O}p_{ns}(\emph{Ch}_{\mathbb{R}})$ is equipped with a model category structure in which weak-equivalences and fibrations are all levelwise. Recall that in $\emph{Ch}_{\mathbb{R}}$ a morphism is a weak-equivalence if it is a quasi-isomorphism, and a fibration if it is an epimorphism.
\end{rem}
Let $\mathcal{O}p_s(\mathcal{C})$ denote the category of symmetric operads in $\mathcal{C}$. Then the obvious forgetful functor $U\colon \mathcal{O}p_s(\mathcal{C}) \longrightarrow \mathcal{O}p_{ns}(\mathcal{C})$ admits a left adjoint $\mbox{Sym}\colon \mathcal{O}p_{ns}(\mathcal{C}) \longrightarrow \mathcal{O}p_s(\mathcal{C})$ defined by
$$\mbox{Sym}(P)(n)=\Sigma_n \otimes P(n)=\coprod_{\sigma \in \Sigma_n} P(n).$$
The following proposition was first proved by Spitzweck \cite{spi01} and by Berger-Moerdijk \cite{ber_moe03} for symmetric operads.
\begin{prop}\label{propaxiom_nonsymoperads} The category of non-symmetric operads in $\mathcal{C}$ satisfies the axiom of relative properness: If $P$ and $Q$ are objects of $\mathcal{O}p_{ns}$ that are cofibrant in $\mathcal{C}^{\mathbb{N}}$,  then the pushout in $\mathcal{O}p_{ns}$ of a weak-equivalence along a cofibration 
$$\xymatrix{P\; \ar@{>->}[r] \ar[d]_{\sim} & R \ar@{.>}[d] \\
          Q \ar@{.>}[r] & S}$$
gives a weak-equivalence $R \stackrel{\sim}{\longrightarrow} S$.         
\end{prop}
\begin{proof}
Since $P$ and $Q$ are cofibrant in the category $\mathcal{C}^{\mathbb{N}}$, it follows that Sym$(P)$ and Sym$(Q)$ are $\Sigma_*$-cofibrant. By applying the functor Sym to the diagram of the statement, we obtain the following pushout diagram in the category of symmetric operads
$$\xymatrix{\mbox{Sym}(P)\; \ar@{>->}[r] \ar[d]_{\sim} & \mbox{Sym}(R) \ar@{.>}[d] \\
          \mbox{Sym}(Q) \ar@{.>}[r] & \mbox{Sym}(S)}.$$

We apply now \cite[Theorem 12.2.B]{fress08} to get a weak-equivalence $\mbox{Sym}(R) \stackrel{\sim}{\longrightarrow} \mbox{Sym}(S)$. Since Sym$(R)(n)$ (respectively Sym$(S)(n)$) is the coproduct over the set $\Sigma_n$ of copies of the object $R(n)$ (respectively$S(n)$), it follows that the morphism $R \longrightarrow S$ in $\mathcal{O}p_{ns}$ is also a weak-equivalence.
\end{proof}
Recall now some necessary definitions.
\begin{defn}\label{multiplicative_def0}
\begin{enumerate}
\item[$\bullet$] A \emph{multiplicative operad} in $\mathcal{C}$ is a couple $(\mathcal{O}, \alpha)$ in which $\mathcal{O}$ is a  non-symmetric operad in $\mathcal{C}$ and  $\alpha \colon \mathcal{A}s \longrightarrow \mathcal{O}$ is a morphism of non-symmetric operads from the associative operad to $\mathcal{O}$. 
\item[$\bullet$] An \emph{up-to-homotopy multiplicative operad} consists of a triple $(\mathcal{O}, \mathcal{A}, \eta)$ in which $\mathcal{O}$ is a  non-symmetric operad in $\mathcal{C}$, $\mathcal{A}$ is an operad weakly equivalent to the associative operad $\mathcal{A}s$, and  $\eta \colon \mathcal{A} \longrightarrow \mathcal{O}$ is a morphism of non-symmetric operads.
\end{enumerate}
\end{defn}  
Notice that by Definition~\ref{multiplicative_def0}, every multiplicative operad is an up-to-homotopy multiplicative operad. Therefore the category of multiplicative operads is a full subcategory of the category of up-to-homotopy multiplicative operads. In the latter category, a morphism from $(\mathcal{O}, \mathcal{A}, \eta)$ to $(\mathcal{O}', \mathcal{A}', \eta')$ consists of morphisms $g \colon \mathcal{O} \longrightarrow \mathcal{O}'$ and  $f \colon \mathcal{A} \longrightarrow \mathcal{A}'$ such that $g \eta=\eta'f$.
\begin{defn}\label{multiplicative_def}
Two multiplicative operads $\mathcal{M}$ and $\mathcal{M}'$ are said to be \emph{weakly equivalent as multiplicative operads} (respectively \emph{weakly equivalent as up-to-homotopy multiplicative operads}) if there is a zig-zag
$$\xymatrix{\mathcal{M} & \mathcal{O}_1 \ar[l]_{\sim} \ar[r]^{\sim} & \cdots & \mathcal{O}_p \ar[l]_{\sim} \ar[r]^{\sim} & \mathcal{M}'}$$
in the category of multiplicative operads (respectively in the category of up-to-homotopy multiplicative operads).
\end{defn} 
We are now ready to state and prove our crucial lemma.
\begin{lem}\label{main_lemma}
In the category $\mathcal{O}p_{ns}$ of non-symmetric operads in $\mathcal{C}$, consider the following commutative diagram
$$\xymatrix{\mathcal{M}_1 & \mathcal{O} \ar[l]_{\sim}^{f_1} \ar[r]^{\sim}_{f_2} & \mathcal{M}_2 \\
\mathcal{A}s  \ar[u]_{\alpha_1} & \mathcal{A} \ar[l]_{\sim}^{\sigma} \ar[r]^{\sim}_{\sigma} \ar[u]_{\eta} & \mathcal{A}s.  \ar[u]_{\alpha_2} }$$
Assume that $\mathcal{A}$ is cofibrant as an object of $\mathcal{C}^{\mathbb{N}}$. Then the operads $\mathcal{M}_1$ and $\mathcal{M}_2$ are weakly equivalent as multiplicative operads.
\end{lem}
\begin{proof}
We begin by the following commutative diagram 
$$\xymatrix{\mathcal{A}s \ar[r]^{\alpha_1} & \mathcal{M}_1\\
              \mathcal{A} \ar[r]^{\eta} \ar[u]^{\sigma}_{\sim} \ar[d]_{\sigma}^{\sim} & \mathcal{O} \ar[u]^{f_1}_{\sim} \ar[d]_{f_2}^{\sim}\\
              \mathcal{A}s \ar[r]^{\alpha_2} & \mathcal{M}_2}$$
Since the object $\mathcal{A}$ is cofibrant in the category $\mathcal{C}^{\mathbb{N}}$, by applying the factorization axiom to the morphism $\eta \colon \mathcal{A} \longrightarrow \mathcal{O}$, we obtain the diagram 

$$\xymatrix{\mathcal{A}s \ar[rr]^{\alpha_1} & & \mathcal{M}_1\\
              \mathcal{A} \ar[rr]^{\eta} \ar[u]^{\sigma}_{\sim} \ar[dd]_{\sigma}^{\sim} \ar@{>->}[rd]^{\eta_1}& &\mathcal{O} \ar[u]^{f_1}_{\sim} \ar[dd]_{f_2}^{\sim}\\ 
       & Y \ar@{->>}[ru]_{\sim}^{\eta_2} & \\
 \mathcal{A}s \ar[rr]^{\alpha_2} & &\mathcal{M}_2} $$
By taking the pushout of the diagram 
$$\xymatrix{\mathcal{A}\;\ar@{>->}[r]^{\eta_1} \ar[d]^{\sim}_{\sigma} & Y \\
                 \mathcal{A}s       }$$
we obtain 
$$\xymatrix{\mathcal{A}s \ar[rr]^{\alpha_1} & & \mathcal{M}_1\\
              \mathcal{A} \ar[rr]^{\eta} \ar[u]^{\sigma}_{\sim} \ar[ddd]_{\sigma}^{\sim} \ar@{>->}[rd]^{\eta_1}& &\mathcal{O} \ar[u]^{f_1}_{\sim} \ar[ddd]_{f_2}^{\sim}\\ 
       & Y \ar@{->>}[ru]_{\sim}^{\eta_2} \ar[d]_{g}^{\sim} & \\
       &\widetilde{\mathcal{O}} \ar@{.>}[rd]& \\
 \mathcal{A}s \ar[rr]^{\alpha_2} \ar[ru]^{\tilde{\eta}} & &\mathcal{M}_2.}$$

Since the operad $\mathcal{A}s $ is cofibrant in $\mathcal{C}^{\mathbb{N}}$(see Remark~\ref{as_is_cofibrant} above) and $\mathcal{A}$ is also cofibrant in $\mathcal{C}^{\mathbb{N}}$ by hypothesis, and since the morphism $\sigma \colon\mathcal{A} \longrightarrow \mathcal{A}s $ is a weak-equivalence and the morphism $\eta_1 \colon \mathcal{A} \longrightarrow Y$ is a cofibration, it follows by Proposition~\ref{propaxiom_nonsymoperads} that the morphism $g \colon Y \longrightarrow \widetilde{\mathcal{O}}$ is a weak-equivalence.\\
Consider now the following pushout diagram 
$$\xymatrix{
    \mathcal{A}\; \ar@{>->}[r]^{\eta_1} \ar[d]_{\sigma}^{\sim} & Y \ar@/^/[rdd]^{f_2\eta_2}_{\sim} \ar[d]_{\sim}^{g} \\
     \mathcal{A}s \ar@/_/[rrd]_{\alpha_2} \ar[r]^{\tilde{\eta}} & \widetilde{\mathcal{O}} \ar@{.>}[rd]\\
     &&\mathcal{M}_2.}$$
     
The universal property of the pushout and the two-out-of-three axiom M2 allow us to obtain a weak-equivalence  
$$\tilde{f}_2 \colon \widetilde{\mathcal{O}} \stackrel{\sim}{\longrightarrow} \mathcal{M}_2.$$   
Similarly, by considering the pushout diagram 
$$\xymatrix{
    \mathcal{A}\; \ar@{>->}[r]^{\eta_1} \ar[d]_{\sigma}^{\sim} & Y \ar@/^/[rdd]^{f_1\eta_2}_{\sim} \ar[d]_{\sim}^{g} \\
     \mathcal{A}s \ar@/_/[rrd]_{\alpha_1} \ar[r]^{\tilde{\eta}} & \widetilde{\mathcal{O}} \ar@{.>}[rd]\\
     &&\mathcal{M}_1.}$$
we deduce the existence of a weak-equivalence $$\tilde{f}_1 \colon \widetilde{\mathcal{O}} \stackrel{\sim}{\longrightarrow} \mathcal{M}_1.$$ 
Finally, we obtain the following commutative  diagram
$$\xymatrix{\mathcal{M}_1 & \widetilde{\mathcal{O}} \ar[l]^{\sim}_{\tilde{f}_1} \ar[r]_{\sim}^{\tilde{f}_2} &  \mathcal{M}_2 \\
\mathcal{A}s \ar[u]_{\alpha_1} & \mathcal{A}s  \ar[l]_{\sim} \ar[r]^{\sim} \ar[u]_{\widetilde{\eta}} & \mathcal{A}s.  \ar[u]_{\alpha_2} }$$    
 
\end{proof}
\section{Formality of the Kontsevich operad as a multiplicative operad}\label{formality_section}
The goal of this section is to prove Theorem~\ref{formality_thm}, Corollary~\ref{main_result} and Corollary~\ref{main_theorem} announced in the introduction. We will also give a very short proof of Theorem~\ref{collapsing_thm}.  The ground field in this section is  $\mathbb{R}$.
\subsection{Proof of Theorem~\ref{formality_thm}}
In \cite{lam_vol}, P. Lambrechts and I. Voli\'c develop the details of Kontsevich's proof \cite{kont99} of the formality of little $N$-disks operad $\mathcal{B}_N=\{\mathcal{B}_N(k)\}_{k \geq 0}$ over the real numbers. 
Note that this formality and its relative version hold in the category $\mathcal{O}p_{ns}(\mbox{Ch}_{\mathbb{R}})$ of non-symmetric operads in chain complexes over $\mathbb{R}$.
\begin{thm}\emph{\cite[Theorem 1.4]{lam_vol}}\label{formality_littledisks}
The little $N$-disks operad is relatively formal over the real numbers when $N \geq 3$. 
\end{thm}
The authors of \cite{lam_vol} prove Theorem~\ref{formality_littledisks} by explicitly constructing  a diagram
\begin{eqnarray}\label{entire_zigzag}
\xymatrix{C_* (\mathcal{K}_N) & C_*(\mathcal{F}_N) \ar[l]_{\sim} \ar[r]^{\sim} & \mathcal{D}^{\vee}_N & H_*(\mathcal{K}_N) \ar[l]_{\sim}\\
              C_*(\mathcal{K}_1) \ar[u] & C_*(\mathcal{F}_1) \ar[l]_{\sim} \ar[r]^{\sim} \ar[u] & \mathcal{D}_1^{\vee} \ar[u] & H_*(\mathcal{K}_1). \ar[l]_{\sim} \ar[u]}
\end{eqnarray}							
In the previous diagram, the quasi-isomorphism $C_*(\mathcal{F}_N) \stackrel{\sim}{\longrightarrow} \mathcal{D}^{\vee}_N$ is built in \cite[Section 9]{lam_vol}, and the quasi-isomorphism $H_*(\mathcal{K}_N)\stackrel{\sim}{\longrightarrow} \mathcal{D}^{\vee}_N$ is built in \cite[Section 8]{lam_vol}.
\begin{proof}[Proof of Theorem~\ref{formality_thm}]
             
In the lower row of (\ref{entire_zigzag}), only $C_*(\mathcal{F}_1)$ is not the associative operad. The other three: $C_*(\mathcal{K}_1)$, $\mathcal{D}_1^{\vee}$, $H_*(\mathcal{K}_1)$ are, and moreover the two morphisms from $C_*(\mathcal{F}_1)$ to $\mathcal{A}s$ are the same, and the objects $C_*(\mathcal{K}_1)$ and $C_*(\mathcal{F}_1)$ are cofibrant in the model category $\mbox{Ch}_{\mathbb{R}}^{\mathbb{N}}$, which implies the desired result by Lemma~\ref{main_lemma}.
\end{proof} 

\subsection{Proof of Corollary~\ref{main_result} and Theorem~\ref{collapsing_thm}}
Recall first that a cosimplicial space $X^{\bullet}$ is said to be formal over $\mathbb{R}$ if the diagram $X^{\bullet} \colon \Delta \longrightarrow \mbox{Top}$ is formal in the sense that $S_*(X^{\bullet}; \mathbb{R})$ and  $H_*(X^{\bullet}; \mathbb{R})$ are weakly equivalent in the category of cosimplicial chain complexes. 
\begin{proof}[Proof of Corollary~\ref{main_result} and Theorem~\ref{collapsing_thm}]
By Theorem~\ref{formality_thm} the operads $S_* (\mathcal{K}_N)$ and $H_* (\mathcal{K}_N)$ are weakly equivalent as multiplicative operads. Therefore the associated cosimplicial objects $(S_* (\mathcal{K}_N))^{\bullet}$ and $(H_* (\mathcal{K}_N))^{\bullet}$  are weakly equivalent in the category of cosimplicial chain complexes over $\mathbb{R}$, hence $S_*(\mathcal{K}_N^{\bullet})$ is formal over $\mathbb{R}$. Now the collapsing of the Bousfield Kan spectral sequence comes from the fact that in the $E^2$ page we can replace the column $S_*(\mathcal{K}_N^p)$ by the homology $H_*(\mathcal{K}_N^p)$, hence the vertical differential vanishes and the spectral sequence collapses (see \cite[Proposition 3.2]{lam_tur_vol10}).
\end{proof}

\subsection{Gerstenhaber algebra structure on the homology of the space of long knots}\label{poisson_section}
The goal here is to prove corollary~\ref{main_theorem} announced in the introduction. 
In \cite{mcc_smith040} McClure and Smith construct an $E_2$ chain operad $\mathcal{T}_2$ that acts on the Hochshild complex $CH(B_*)$, when $B_*$ is an operad with multiplication in chain complexes (recall that an \textit{$E_2$ chain operad} is a chain operad weakly equivalent to the normalized singular chain of the little $2$-cubes operad). This action induces a Gerstenhaber algebra structure on the Hochshild homology  $HH(B_*)$. It is very important to note that this structure coincides with the natural one \cite{ger_vor95} because $\mathcal{T}_2$ is a solution of Deligne's conjecture.
Let $\mathcal{T}_2$-algebras denote the category of chain complexes equipped with an action of the operad $\mathcal{T}_2$. Let $\mathcal{O}p_*(\mbox{Ch}_{\mathbb{R}})$ denote the category of multiplicative non-symmetric operads in chain complexes over real numbers. 
\begin{lem}\emph{\cite{mcc_smith040}}\label{mcc_smith_lemma}
There exists a functor 
$$CH \colon \mathcal{O}p_*(\emph{Ch}_{\mathbb{R}}) \longrightarrow \mathcal{T}_2\emph{-algebras}$$
that preserves weak-equivalences.
\end{lem}   
We are now ready to prove Corollary~\ref{main_theorem}.
\begin{proof}[Proof of Corollary~\ref{main_theorem}]
First, in the category of multiplicative operads in chain complexes over $\mathbb{R}$, we have by Theorem~\ref{formality_thm} a zig-zag
\begin{eqnarray}\label{zigzag_formality}
\xymatrix{S_*(\mathcal{K}_N) & \ar[l]_-{\sim} \cdots \ar[r]^-{\sim} & H_*(\mathcal{K}_N).}
\end{eqnarray}
Next, by applying the normalized Hochshild complex functor $CH$ to (\ref{zigzag_formality}), we obtain a zig-zag 
\begin{eqnarray}\label{zigzag_poisson}
\xymatrix{CH(S_*(\mathcal{K}_N)) & \ar[l]_-{\sim} \cdots \ar[r]^-{\sim} & CH(H_*(\mathcal{K}_N))}
\end{eqnarray}
in the category of $\mathcal{T}_2$-algebras by Lemma~\ref{mcc_smith_lemma}. Therefore the homology of (\ref{zigzag_poisson}) gives 
\begin{eqnarray}\label{zigzag_poisson2}
\xymatrix{HH(S_*(\mathcal{K}_N)) & \ar[l]_-{\cong} \cdots \ar[r]^-{\cong} & HH(H_*(\mathcal{K}_N))},
\end{eqnarray}
which respects the Gerstenhaber algebra structure induced by the $H_*(\mathcal{T}_2)$ action.
Finally, the desired result follows from (\ref{zigzag_poisson2}) and Remark~\ref{sakai_rem}.
\end{proof}

%
%
%
%


\begin{thebibliography}
\bibitem{ber_moe03} C. Berger, I. Moerdijk, Axiomatic homotopy theory for operads, \textit{Comment. Math. Helv.}, vol. 78 (2003), no. 4, 805-831.
\bibitem{bud07} R. Budney, Little cubes and long knots, \textit{Topology} 46 (2007), no. 1, 1-27.
\bibitem{dwyer_hess10} W. Dwyer, K. Hess. Long knots and maps between operads, \textit{Geom. Topol.}, vol.  16 (2012), no. 2, 919-955.
\bibitem{fress08} B. Fresse, Modules over Operads and Functors, Lectures Notes in Mathematics, vol. 1967, \textit{Springer-Verlag, Berlin} (2009).
\bibitem{ger_vor95} M. Gerstenhaber, A. A. Voronov, Homotopy G-algebras and moduli space operad, \textit{Internat. Math. Res. Notices} 1995, no. 3, 141-153 (electronic).
\bibitem{hovey99}  M. Hovey, Model categories, Mathematical Surveys and Monographs, vol. 63, \textit{American Mathematical Society, Providence, RI}, 1999.
\bibitem{kont99} M. Kontsevich, Operads and motives in deformation quantization, Moshé Flato (1937-1998), \textit{Lett. Math. Phys.} 48 (1999), no. 1, 35-72.
\bibitem{lam_tur_vol10} P. Lambrechts, V. Turchin and I. Voli\'c, The rational homology of spaces of long knots in codimension $>$ $2$, \textit{Geom. Topol.} 14 (2010), no.4, 2151--2187.
\bibitem{lam_vol} P. Lambrechts and I. Voli\'c, Formality of the little N-disks operad, arXiv:0808.0457v3 (2012).
\bibitem{may72} P. May, The geometry of iterated loop spaces, Lectures Notes in Mathematics, Vol. 271, \textit{Springer-Verlag, Berlin-New York}, 1972.
\bibitem{mcc_smith04} JE.McClure and JH. Smith, Cosimplicial objects and little n-cubes I, \textit{Amer. J . Math.}, 126 (2004), no. 5,  1109-1153.
\bibitem{mcc_smith040} JE McClure, JH Smith. Operads and cosimplicial objects: an introduction. \textit{Axiomatic, enriched and motivic homotopy theory}, 133-171, NATO Sci. Ser. II Math. Phys. Chem., 131, Kluwer Acad. Publ., Dordrecht, 2004.
\bibitem{moriya12} S. Moriya, Sinha's spectral sequence and homotopical algebra of operads, arXiv:1210.0996v2 (2012).
\bibitem{muro11} F. Muro, Homotopy theory of nonsymmetric operads, \textit{Algebr. Geom. Topol.} 11 (2011), no. 3,  1541-1599.
\bibitem{sakai08} K. Sakai, Poisson structures on the homology of the space of knots, \textit{Groups, homotopy and configuration spaces}, 463-482, Geom. Topol. Monogr., 13, \textit{Geom. Topol. Publ., Coventry}, 2008.
\bibitem{sal06} P. Salvatore, Knots, operads and double loop spaces, \textit{Int. Math. Res. Not.} 2006, Art. ID 13628, 22 pp.
\bibitem{sin06} D. Sinha, Operads and knot spaces, \textit{J. Amer. Math. Soc.}, 19 (2006), no.2, 461-486 (electronic).
\bibitem{spi01} M. Spitzweck, Operads, Algebras and Modules in General Model Categories, preprint arXiv:math.AT/ 0101102 (2001).
\bibitem{tur10} V. Turchin, Delooping totalization of a multiplicative operad, arXiv:1012.5957v2 (2010).
\end{thebibliography}
\end{document}